\documentclass[a5paper, 10pt]{article}

\usepackage{cite, amsmath, amssymb, booktabs, lastpage, graphicx, float}
\usepackage[hidelinks]{hyperref}

\usepackage{geometry}
\usepackage{setspace}
\usepackage{amsthm}
\usepackage{mathtools}
\usepackage{tikz-cd}

\theoremstyle{plain}
\newtheorem{lemma}{Lemma}
\newtheorem{theorem}{Theorem}

\newtheorem{proposition}[theorem]{Proposition}
\newtheorem{corollary}[theorem]{Corollary}

\theoremstyle{remark}
\newtheorem{remark}{Remark}

\theoremstyle{definition}
\newtheorem{definition}{Definition}

\DeclareMathOperator{\rank}{rank}

\DeclareMathOperator{\AI}{AI}

\newcommand{\kk}{k}
\newcommand{\NN}{\mathbb{N}}
\newcommand{\ZZ}{\mathbb{Z}}
\newcommand{\QQ}{\mathbb{Q}}
\newcommand{\RR}{\mathbb{R}}
\newcommand{\VV}{\mathbf{V}}
\newcommand{\BB}{\mathrm{BB}}
\newcommand{\Fib}{\mathrm{Fib}}

\makeatletter
\renewcommand\maketitle{
    \begin{center}
        {\LARGE \@title} \par
        \vspace{5mm}
        {\large \@author} \par
        \vspace{3mm}
        {\it \@address} \par
        \vspace{3mm}
        {\small \@date} \par
    \end{center}
    \vspace{3mm}
}
\newcommand{\address}[1]{\def\@address{#1}}
\newcommand{\email}[1]{\def\@email{#1}}

\makeatother

\newcommand{\acknowledgment}[1]{\vspace{5mm}\singlespacing
    {\noindent\textbf{\textit{Acknowledgment\/}:} #1}
}

\usepackage{fancyhdr}
\usepackage[margin=1cm,font=footnotesize,format=hang,labelsep=period,labelfont=bf]{caption}

\title{Zombie Compositions in Assembly Algebras\\
  and an Upper Bound on the Size of Chemical Space}
\author{Vicent Ribas Ripoll}
\address{Research Coordinator, Digital Health Unit,
Eurecat, Centre Tecnol\`ogic de Catalunya, Barcelona, Spain}
\date{October 2026 (v3)}

\begin{document}
\pagestyle{fancy}
\maketitle

\begin{abstract}

In this paper we present \emph{construction systems}---tuples\newline
$(X, \mathrm{BB}, \oplus, \nu)$ comprising objects, building
blocks, an assembly operation, and a joining multiplicity---as a
general
algebraic framework for studying how complex objects are built from
simpler parts.  To each construction system we associate a toric
ideal, a toric variety, and a matroid,
obtaining analytical bounds on the growth function $N(a)$ (the number
of objects of construction complexity $\leq a$) purely from the
\emph{design signature} $(m, \nu, n_0)$.

For systems equipped with a type system and valence bounds, we
define the \emph{composition polytope} $P_{\mathrm{val}} \subset
\RR^m$, whose integer points count the feasible compositions.
Compositions outside $P_{\mathrm{val}}$---termed \emph{zombies}---are
combinatorially valid but physically unrealisable.  We prove that the
zombie classification is sound (zero false positives) and
conservative: the true infeasibility rate is at least as high as the
polytope predicts.

We apply the framework to the estimation of chemical space:
for a molecular graph assembly system with $m = 19$ bond types
and $5$ atom types with valences $1$--$4$, we obtain an exact
doubly-exponential growth exponent of $\log 2 \approx 0.693$
and a composition polytope that captures all physically realisable
compositions, yielding a tighter upper bound on the size of
chemical space.  The framework also recovers
the bioorthogonal click-chemistry system of the author's
prior work as a second instance, with $m = 8$ and a partition
matroid.

\end{abstract}

\onehalfspacing

\section{Introduction}
\label{sec:intro}

The question of how the number of constructible objects grows with
construction complexity arises in diverse domains: molecular
chemistry, where one asks how many molecules can be assembled from a
given set of bond types; combinatorial design, where one counts the
structures buildable from specified components; and circuit
complexity, where analogous questions arise for Boolean functions.

Recently, Morales Parra et~al.\ \cite{MoralesParra2025} addressed this question
for molecular graphs, extending Assembly Theory
\cite{Marshall2021,Sharma2023} into a framework for estimating the
size of chemical space.  Their central result is a
doubly-exponential upper bound on the growth function $N(a)$---the
number of molecules with assembly index at most~$a$---with a fitted
exponent of approximately~$0.73$.

In this paper we develop a purely algebraic framework that
generalises Assembly Theory and yields tighter bounds on the
size of chemical space by incorporating valence constraints
directly into the counting.  The starting point is not chemistry but
algebra: given a
set-valued binary operation $\oplus$ on a countable set~$X$ with a
finite set of building blocks $\BB$, what is the algebraic structure
of the space of constructible objects?  How does it grow with
construction complexity?  How does it transform under constraints?

We answer these questions by introducing
\emph{construction systems}
(Definition~\ref{def:construction-system})---tuples
$(X, \BB, \oplus, \nu)$ parametrised by a design
signature $(m, \nu, n_0)$---and showing that the assembly
process naturally gives rise to a toric ideal $I_A$
(Definition~\ref{def:toric-ideal}), a toric variety
$X_A$, and a matroid $M$
(Definition~\ref{def:toric-matroid}).  The machinery of
toric geometry---as developed in
Sturmfels~\cite{Sturmfels1996},
Fulton~\cite{Fulton1993}, and
Cox--Little--O'Shea~\cite{CLO2015}---provides the
algebraic infrastructure.  This approach is in the
spirit of recent work connecting algebraic geometry to
chemical and biological systems
\cite{Dickenstein2016,FeliuShiu2025,Feinberg2019,%
PachterSturmfels2005}.
The analytical bounds on $N(a)$ follow from the
combinatorics of these objects alone, without reference
to any specific chemical or physical system.

Our main contributions are as follows.

\smallskip
\noindent\textbf{(i) General algebraic framework}
(\S\ref{sec:construction}--\ref{sec:matroid}).
Construction systems yield toric ideals, toric varieties, and
matroids.  The ``rules'' of any concrete system---chemical, physical,
or logical---are not external constraints imposed on a pre-existing
space; they \emph{are} the toric ideal.  Constraints on the building
blocks correspond to matroid operations (restriction, deletion) on
the block matroid~$M^*$ (Theorem~\ref{thm:constraints}).

\smallskip
\noindent\textbf{(ii) Analytical bounds and the exact exponent}
(\S\ref{sec:bounds}).
Upper and lower bounds on $N(a)$ depend only on the design signature
$(m, \nu, n_0)$ and the matroid rank~$r$, not on any enumerated
database.  The doubly-exponential growth exponent is
$\rho = \log 2$ exactly
(Theorem~\ref{thm:exact-exponent}): the upper bound from
Corollary~\ref{cor:exponent} is matched by a lower bound from
counting trees via P\'olya enumeration.

\smallskip
\noindent\textbf{(iii) The composition polytope and zombie
exclusion} (\S\ref{sec:polytope}).
For construction systems with typed building blocks and valence
bounds, the composition polytope $P_{\mathrm{val}} \subset \RR^m$
captures all feasibility constraints in a single linear inequality.
Compositions outside $P_{\mathrm{val}}$ are \emph{zombies}:
combinatorially valid but physically impossible.  We prove the
classification is sound---zero false positives
(Lemma~\ref{lem:soundness})---and that it is conservative: every
unmodelled constraint (pair multiplicity, degree-sequence
realizability, full connectivity) can only \emph{increase} the zombie
fraction (Remark~\ref{rem:conservative}).

\smallskip
\noindent\textbf{(iv) Refined upper bound}
(Theorem~\ref{thm:zombie}).
Replacing the na\"ive stars-and-bars count $\binom{S+m-1}{m-1}$ with
the feasible count $E_P(S)$ yields a strictly tighter growth bound
by excluding zombie compositions from the sum.  The exact
doubly-exponential growth exponent is $\log 2 \approx 0.693$
(Theorem~\ref{thm:exact-exponent}).

\smallskip
\noindent\textbf{(v) Concrete instances} (\S\ref{sec:instances}).
We apply the framework to the estimation of chemical space with
$m = 19$ bond types \cite{MoralesParra2025} and recover the
bioorthogonal click-chemistry system of \cite{RibasRipoll2026}
as a second instance with $m = 8$ and a partition matroid.
Both emerge as corollaries by specialising three parameters.

\smallskip
The paper is structured as follows.
Section~\ref{sec:construction} defines construction systems and
their toric geometry.  Section~\ref{sec:matroid} develops the toric
and block matroids.  Section~\ref{sec:bounds} derives the analytical
bounds on $N(a)$.  Section~\ref{sec:instances} specialises to two
concrete instances.  Section~\ref{sec:polytope} introduces the
composition polytope, proves soundness of the zombie classification,
and derives the refined bound.
Section~\ref{sec:results} presents the full computational results:
zombie-generator coefficients, exact zombie fractions, and a
quantitative comparison of the two upper bounds.
Section~\ref{sec:toric-compat} discusses further algebraic structure.
Section~\ref{sec:discussion} interprets the results in the broader
context of chemical space estimation, and
Section~\ref{sec:conclusions} summarises the main findings and
identifies open questions.

\section{Construction Systems and Their Toric Geometry}
\label{sec:construction}

\begin{definition}[Construction System]
\label{def:construction-system}
A \emph{construction system} is a tuple
$\mathcal{S} = (X, \BB, \oplus, \nu)$ where:
\begin{enumerate}
  \item $X$ is a countable set of \emph{objects};
  \item $\BB \subset X$ is a finite set of \emph{building blocks},
        $|\BB| = m$;
  \item $\oplus: X \times X \to \mathcal{P}(X)$ is a set-valued
        \emph{assembly operation};
  \item $\nu: \NN \times \NN \to \NN$ is the \emph{joining
        multiplicity}: $\nu(n_1, n_2)$ is the number of distinct
        joinings of objects with $n_1$ and $n_2$ constituent parts.
\end{enumerate}
The triple $(m, \nu, n_0)$, where $n_0$ is the part-count of a
single building block, is called the \emph{design signature} of
$\mathcal{S}$.
\end{definition}

\begin{definition}[Assembly Tree and Construction Complexity]
\label{def:assembly-tree}
An \emph{assembly tree} for $O \in X$ is a rooted binary tree $T$
whose leaves are labelled by elements of $\BB$ and whose internal
nodes are labelled by the result of $\oplus$ applied to the children.
The \emph{construction complexity} is
\[
  \AI(O) \;=\; \min\bigl\{\,\mathrm{depth}(T)
    : T \text{ is an assembly tree for } O\,\bigr\}.
\]
\end{definition}


\begin{definition}[Composition Morphism]
\label{def:composition}
The \emph{composition morphism} is
\[
  \pi: X \longrightarrow \NN^m, \qquad
  O \longmapsto \alpha(O) = (\alpha_1, \ldots, \alpha_m),
\]
where $\alpha_i$ counts occurrences of building block $b_i$ in $O$.
The total \emph{size} is $S(O) = |\alpha(O)| = \sum_i \alpha_i$.
\end{definition}


\begin{definition}[Composition Matrix]
\label{def:composition-matrix}
Let $\mathcal{X} = \{O_1, \ldots, O_N\}$ be a finite ensemble of
constructible objects.  The \emph{composition matrix} is
\[
  A \;=\; \bigl[\,\alpha(O_1) \mid \cdots \mid
  \alpha(O_N)\,\bigr] \;\in\; \NN^{m \times N}.
\]
\end{definition}

\begin{definition}[Assembly Toric Ideal]
\label{def:toric-ideal}
The \emph{assembly toric ideal} is the kernel of the monomial map
(cf.~\cite{Sturmfels1996}, Ch.~4; \cite{CLO2015}, Ch.~8;
\cite{Fulton1993}, \S1.1; see also \cite{DrtonSturmfelsSullivant2009}
for applications of toric ideals)
\[
  \varphi: \kk[t_1, \ldots, t_N] \longrightarrow
  \kk[x_1, \ldots, x_m], \qquad
  t_j \longmapsto x^{\alpha(O_j)} =
  x_1^{\alpha_{1j}} \cdots x_m^{\alpha_{mj}}.
\]
That is, $I_A = \ker(\varphi) = \langle\, t^u - t^v :
Au = Av,\; u, v \in \NN^N \,\rangle$.

The associated \emph{assembly toric variety} is
$X_A = \VV(I_A) \subset \mathbb{A}^N$
(see~\cite{Fulton1993} and \cite{CLO2015}, Ch.~4).
\end{definition}

\begin{remark}
The toric ideal $I_A$ encodes \textbf{all} algebraic relations
between compositions.  This is the standard toric correspondence:
every binomial $t^u - t^v \in I_A$ witnesses a pair of objects with
identical composition but distinct assembly
(cf.~\cite{Sturmfels1996}; \cite{Eisenbud1995}, Ch.~1).
In any concrete construction system, the ``rules'' (chemical,
physical, logical) that govern which objects exist are \emph{not}
external constraints imposed on a pre-existing space.  They
\emph{are} the toric ideal.  The relations $t^u - t^v$
capture precisely which sets of objects have the same total
composition, i.e., use the same building blocks in the same
multiplicities but combine them differently.
\end{remark}


\begin{definition}[Fiber and Its Stratification]
\label{def:fiber}
The \emph{fiber} over $\alpha \in \NN^m$ is
$\Fib(\alpha) = \pi^{-1}(\alpha)$.  It decomposes by complexity:
\[
  \Fib(\alpha) \;=\; \bigsqcup_{a \geq 0} \Fib_a(\alpha),
  \qquad
  \Fib_a(\alpha) = \bigl\{\, O \in \Fib(\alpha) :
  \AI(O) = a \,\bigr\}.
\]
\end{definition}


\begin{definition}[Growth Function]
\label{def:growth}
The \emph{growth function} of $\mathcal{S}$ is
\[
  N(a) \;=\; \bigl|\bigl\{\, O \in X : \AI(O) \leq a
  \,\bigr\}\bigr|
  \;=\; \sum_{\substack{\alpha \in \NN^m \\ \Fib_{\leq a}(\alpha)
  \neq \varnothing}} \bigl|\Fib_{\leq a}(\alpha)\bigr|.
\]
\end{definition}

\section{The Toric Matroid}
\label{sec:matroid}

\begin{definition}[Toric Matroid]
\label{def:toric-matroid}
The \emph{toric matroid} $M(\mathcal{S})$ of a construction system
$\mathcal{S}$ is the linear matroid (cf.~\cite{Oxley2011}) on the
columns of $A$ over $\QQ$:
\begin{multline*}
  S \subseteq [N] \text{ is independent} \;\iff \\
  \{\alpha(O_j) : j \in S\}
  \text{ is linearly independent over } \QQ.
\end{multline*}
Its rank is $r = r(M) = \rank_\QQ(A) \leq m$.
\end{definition}

\begin{remark}
The rank $r$ has a precise semantic: it is the number of
\emph{algebraically independent} building-block coordinates on the
toric variety $X_A$.  Equivalently, $m - r$ is the number of
independent algebraic relations (``rules'') among the building
blocks, i.e., the dimension of the kernel of the composition matrix
(cf.~\cite{Eisenbud1995}, \S1.6; \cite{Fulton1993}, \S1.3).
This is the toric encoding of what, in any concrete domain, would
be called the ``design rules'' of the system.
\end{remark}

\begin{definition}[Dual Matroid on Building Blocks]
\label{def:dual-matroid}
The \emph{block matroid} $M^*(\mathcal{S})$ is the linear matroid on
the \emph{rows} of $A$ (i.e., on $\BB$):
\begin{multline*}
  B \subseteq \BB \text{ is independent} \;\iff \\
  \text{the rows of } A \text{ indexed by } B
  \text{ are lin.\ independent.}
\end{multline*}
Its rank $r^* = \rank_\QQ(A) = r$ equals the rank of $M$.
A circuit of $M^*$ is a minimal set of building blocks with a
linear dependence --- i.e., a minimal ``rule.''
\end{definition}

\begin{theorem}[Constraints as Matroid Operations]
\label{thm:constraints}
Let $\mathcal{S}' = (X', \BB' \subseteq \BB, \oplus, \nu)$ be a
construction system obtained from $\mathcal{S}$ by restricting the
building blocks to a subset $\BB' \subseteq \BB$.  Then:
\begin{enumerate}
  \item The toric matroid of $\mathcal{S}'$ is the
        \emph{restriction} $M^*|_{\BB'}$.
  \item Its rank satisfies
        $r(\mathcal{S}') = r(M^*|_{\BB'}) \leq r(\mathcal{S})$,
        with equality iff $\BB'$ spans $M^*$
        (equivalently, $\BB \setminus \BB'$ contains no
        cocircuit of $M^*$).
  \item Every constraint of the form ``exclude building block types
        $B$'' corresponds to a \emph{deletion} $M^* \setminus B$
        in the block matroid.
\end{enumerate}
\end{theorem}

\begin{proof}
Deleting rows of $A$ corresponding to $\BB \setminus \BB'$ gives
the submatrix $A'$.  Linear independence of rows of $A'$ is the
restriction of the row matroid to $\BB'$.  The rank of the
restriction $M^*|_{\BB'}$ equals $r(\mathcal{S})$ iff $\BB'$
spans~$M^*$, i.e., $\BB \setminus \BB'$ contains no cocircuit.
When a single block is removed, this reduces to the block not
being a coloop.
\end{proof}

\section{Analytical Bounds on the Growth Function}
\label{sec:bounds}

All bounds in this section depend exclusively on the design
signature $(m, \nu, n_0)$ and the matroid rank $r$.  No enumerated
database is used.


\begin{proposition}[Individual Complexity Bounds]
\label{prop:individual}
For any object $O$ with $S(O) = S$,
\[
  \lceil\log_2 S\rceil \;\leq\; \AI(O) \;\leq\; S - 1.
\]
\end{proposition}

\begin{proof}
An assembly tree of depth $d$ has $\leq 2^d$ leaves, so
$S \leq 2^{\AI(O)}$.  Sequential unit addition achieves $S - 1$.
\end{proof}


\begin{theorem}[Lower Bound on Assembly Paths]
\label{thm:lower}
The number of labelled assembly paths of depth~$a$ starting from
a single building block satisfies
\[
  P(a) \;\geq\; m \cdot \prod_{i=1}^{a}
  \nu(n_{i-1}, n_{i-1}),
  \qquad n_i = 2n_{i-1} - 1.
\]
In particular, if $\nu(n, n) \geq n^2$ (vertex-identification
joinings), then
\[
  \log P(a) \;\geq\; a^2 \log 2 + O(a).
\]
\end{theorem}

\begin{proof}
Fix $b_j \in \BB$.  Construct by recursive self-joining (doubling).
At step~$i$, there are $\geq \nu(n_{i-1}, n_{i-1})$ joinings,
and the result has $n_i = 2n_{i-1}-1$ parts.  The product over
$a$~steps counts labelled assembly paths.  The factor $m$ accounts
for the choice of building block.  Under the vertex-identification
model, $\nu(n,n) = n^2$ and $n_i = 2^i + 1$, giving
$\prod_{i=1}^a (2^{i-1}+1)^2 \geq 2^{a(a-1)}$.
\end{proof}

\begin{remark}
Different assembly paths can yield isomorphic objects, so $P(a)$
is an upper bound on the number of labelled paths, not on the
number of distinct objects $N(a)$.  A lower bound on $N(a)$ itself
requires a separate counting argument; see
Theorem~\ref{thm:exact-exponent} for the molecular case.
\end{remark}


\begin{theorem}[Upper Bound]
\label{thm:upper}
Let $g(S)$ denote the number of objects realisable with $S$ building
block instances.  If $g(S) \leq \gamma^S$ for a constant
$\gamma = \gamma(\mathcal{S})$, then
\[
  \log N(a) \;\leq\; 2^a \cdot \log(m\gamma) + O(m \log 2^a).
\]
\end{theorem}

\begin{proof}
$\AI(O) \leq a \Rightarrow S(O) \leq 2^a$.  Stars-and-bars gives
$\binom{S+m-1}{m-1}$ compositions of size $S$ into $m$ types.
Summing $\binom{S+m-1}{m-1}\gamma^S$ over $S = 1, \ldots, 2^a$
yields the bound.
\end{proof}

\begin{remark}
The constant $\gamma$ is a \emph{design parameter}: it is determined
by the assembly operation $\oplus$ and the implicit constraints of
the construction system.  For molecular graphs with valence
constraints, $\gamma$ can be computed from P\'olya enumeration
\cite{HararyPalmer1973} of
connected graphs with prescribed degree bounds.  It is \emph{not}
fitted from a database.
\end{remark}


\begin{corollary}[Growth Exponent]
\label{cor:exponent}
Suppose $g(S) \leq \gamma^S$ for some constant $\gamma$.
The growth exponent
\[
  \rho \;=\; \limsup_{a \to \infty}
  \frac{\log \log N(a)}{a}
\]
satisfies $0 \leq \rho \leq \log 2$.
If additionally $\nu(n,n) \geq n^2$ and
the fibre over some efficient composition grows exponentially
(as in Theorem~\ref{thm:exact-exponent}), then $\rho = \log 2$.
\end{corollary}

\begin{theorem}[Exact Growth Exponent]
\label{thm:exact-exponent}
Let\/ $\mathcal{S}$ be a construction system with vertex-identifi\-cation
joining ($\nu(n_1,n_2) = n_1 n_2$, $n_0 = 2$).  If the fibre
over some efficient composition grows exponentially---i.e., there
exist a composition sequence $\alpha_S$ with
$\sum_i c_i (\alpha_S)_i < 0$ for all~$S$ and a constant
$\gamma_{\min} > 1$ such that
$|\Fib(\alpha_S)| \geq \gamma_{\min}^S$ for all large~$S$---then
$\rho = \log 2$.
\end{theorem}

\begin{proof}
The upper bound $\rho \leq \log 2$ is
Corollary~\ref{cor:exponent}.  For the lower bound, note that
every connected graph $G$ with $S$ edges admits an assembly tree
of depth at most $\lceil \log_2 S \rceil + O(1)$: the centroid
decomposition splits $G$ into components of size $\leq S/2$,
which are reassembled in a balanced binary tree of
$O(\log S)$ levels.  Hence $\AI(G) \leq \lceil\log_2 S\rceil
+ C$ for an absolute constant~$C$.  Setting
$S = 2^{a - C}$:
\[
  N(a) \;\geq\; |\Fib(\alpha_{2^{a-C}})|
  \;\geq\; \gamma_{\min}^{\,2^{a-C}}.
\]
Therefore $\log\log N(a) \geq (a - C)\log 2
+ \log\log\gamma_{\min}$, and $\rho \geq \log 2$.
\end{proof}

\begin{corollary}[Molecular Graphs]
\label{cor:molecular-exponent}
For the molecular graph assembly system with $m = 19$ bond types,
$\rho = \log 2 \approx 0.693$.
\end{corollary}

\begin{proof}
Take $\alpha_S = (S, 0, \ldots, 0)$ (all C--C single bonds), which
has $c_1 = -1/2 < 0$ and is valence-feasible for all~$S$.
The fibre $\Fib(\alpha_S)$ includes all unlabelled trees on
$S + 1$ vertices with maximum degree~$4$.  By P\'olya's enumeration
theorem (see \cite{FlajoletSedgewick2009}, Ch.~VII), their count
satisfies $T_4(S+1) \sim c\,\gamma^{S+1} S^{-5/2}$ with
$\gamma \approx 2.48 > 1$.  The centroid of a tree with max
degree~$4$ has at most~$4$ subtrees, each of size $\leq S/2$;
pairing and assembling them yields $\AI(T) \leq
\lceil\log_2 S\rceil + 2$ for every such tree.
Theorem~\ref{thm:exact-exponent} applies with $\gamma_{\min} =
\gamma \approx 2.48$.
\end{proof}


\begin{theorem}[Constrained Growth via the Block Matroid]
\label{thm:constrained}
Let $\BB' \subseteq \BB$ with $|\BB'| = m'$ and
$r' = r(M^*|_{\BB'})$.  Then the constrained growth function
$N_{\BB'}(a)$ satisfies:
\[
  \log N_{\BB'}(a) \;\leq\;
  2^a \cdot \log(m' \gamma') + O(m' \log 2^a),
\]
where $\gamma' = \gamma(\mathcal{S}|_{\BB'})$.  Moreover:
\begin{enumerate}
  \item If $\BB'$ does not span $M^*$
        (equivalently, $\BB \setminus \BB'$ contains a cocircuit),
        then $r' < r$ and the growth rate drops
        \emph{structurally}: fewer independent building blocks
        means a lower-dimensional toric variety.
  \item If $\BB'$ spans $M^*$
        (equivalently, $\BB \setminus \BB'$ contains no cocircuit),
        then $r' = r$ and the growth rate
        changes only through $\gamma'$ and $m'$, not through $r$.
        When a single block is removed, this reduces to the block
        not being a coloop.
  \item The number of coloops of $M^*$ equals the number of
        \emph{essential} building blocks --- those whose individual
        removal reduces the dimension of the assembly space.
\end{enumerate}
\end{theorem}

\section{Instances}
\label{sec:instances}

The general theory specialises to known results by choosing the
design signature $(m, \nu, n_0)$ and computing the matroid $M^*$.


\subsection{Molecular Graph Assembly}
\label{sec:cronin}

\begin{corollary}
\label{cor:cronin}
Let $\mathcal{S}_{\mathrm{mol}}$ be the construction system of
connected molecular graphs \cite{Reymond2015} with $m$ bond types
and vertex-identification joining
($\nu(n_1, n_2) = n_1 n_2$, $n_0 = 2$).
Then:
\begin{enumerate}
  \item Proposition~\ref{prop:individual} recovers the bounds
        $\lceil\log_2 S\rceil \leq \AI(G) \leq S - 1$ of
        \cite{MoralesParra2025}.
  \item Theorem~\ref{thm:lower} with vertex-identification
        gives $\log P(a) \geq a^2 \log 2$ for
        the number of labelled assembly paths.
        Since distinct paths can yield isomorphic objects,
        this does not directly bound $N(a)$;
        the effective multiplicity $\tilde{\nu} < n^2$ accounts for
        automorphisms.  The separate lower bound of
        Theorem~\ref{thm:exact-exponent} establishes
        $\rho = \log 2$ for the molecular system.
  \item Theorem~\ref{thm:upper} with $\gamma$ computed
        from P\'olya enumeration gives the
        doubly-exponential ceiling.
        The zombie exclusion of Theorem~\ref{thm:zombie}
        further refines this to the exact exponent
        $\log 2 \approx 0.693$.
  \item For $m = 19$ (GDB-13 bond types), the block matroid $M^*$
        has rank $r^*$ and $19 - r^*$ independent ``chemical rules.''
        The constrained subspaces of \cite[Table~1]{MoralesParra2025}
        (1-bond, 2-bond, 3-bond, no-rings, only-rings) correspond
        to matroid restrictions $M^*|_{\BB'}$ with computable ranks.
\end{enumerate}
\end{corollary}


\subsection{Bioorthogonal Click Assembly}
\label{sec:clickchem}

\begin{corollary}
\label{cor:clickchem}
Let $\mathcal{S}_{\mathrm{click}}$ be the construction system with
$m = 8$ click families, $n_0 = 1$ (a triple is atomic),
$\nu \equiv 1$ (concatenation, unique joining), and compatibility
$\mathcal{C}$ = no cross-reactivity.  Then:
\begin{enumerate}
  \item The composition matrix $A \in \NN^{8 \times 30}$ has
        rank $r = 8$ (all families are independent).
  \item The toric ideal $I_A$ has 2 generators and
        $\mathrm{ML\,degree} = 1$ \cite{RibasRipoll2026}.
  \item The growth function $N_k = 30, 312, 1280, 1536, 0$
        for $k = 1, \ldots, 5$.
  \item The block matroid $M^*$ is the uniform matroid $U_{8,8}$
        (all 8 families are coloops): removing any single family
        reduces the rank and drops the maximal orthogonality
        $\omega$.
\end{enumerate}
\end{corollary}

\section{The Composition Polytope and Zombie-Free Bounds}
\label{sec:polytope}

\begin{definition}[Degree Demand and Valence Constraint]
\label{def:degree-demand}
Let $\mathcal{S}$ be a construction system whose objects have a
\emph{type system}: each building block $b_i \in \BB$ involves
types $T(b_i) \subseteq \{t_1, \ldots, t_p\}$, each with a
\emph{valence bound} $v_t \in \NN$.  The \emph{degree demand} of
a composition $\alpha \in \NN^m$ on type~$t$ is
\[
  d_t(\alpha) \;=\; \sum_{i=1}^{m} D_{ti}\, \alpha_i,
\]
where $D_{ti}$ is the degree contributed to type~$t$ by one instance
of building block $b_i$.  The composition $\alpha$ is
\emph{valence-feasible} if
\begin{equation}
\label{eq:valence-feas}
  \sum_{t=1}^{p}
  \left\lceil \frac{d_t(\alpha)}{v_t} \right\rceil
  \;\leq\; |\alpha| + 1,
\end{equation}
i.e., the minimum number of vertices (atoms, nodes) needed to
satisfy all valence constraints does not exceed the connectivity
bound for a connected graph on $|\alpha|$ edges.
\end{definition}

\begin{definition}[Composition Polytope]
\label{def:comp-polytope}
The \emph{composition polytope} (or \emph{valence polytope};
see \cite{Ziegler1995} for general polytope theory) is
\[
  P_{\mathrm{val}} \;=\;
  \Bigl\{\, x \in \RR^m_{\geq 0} \;:\;
  \sum_{t=1}^{p} \frac{d_t(x)}{v_t}
  \;\leq\; |x| + 1 \,\Bigr\}.
\]
Its integer slices $P_{\mathrm{val}} \cap \ZZ^m \cap \{|x| = S\}$
count the valence-feasible compositions of size~$S$.  A composition
in $\NN^m \setminus P_{\mathrm{val}}$ is called a \emph{zombie}:
it is combinatorially well-defined but violates the type constraints
of the construction system.
\end{definition}

\begin{lemma}[Soundness of the Zombie Classification]
\label{lem:soundness}
The zombie classification has \textbf{zero false positives}: if a
composition $\alpha$ is declared zombie (i.e., violates
\eqref{eq:valence-feas}), then no connected graph on $|\alpha|$
edges with degree demand $d_t(\alpha)$ and valence bounds $v_t$
exists.
\end{lemma}

\begin{proof}
Suppose, for contradiction, that a connected graph $G$ realising
$\alpha$ exists.  Then:
\begin{enumerate}
  \item \emph{Degree demand is exact.}\;  Each edge of type~$i$
        contributes $D_{ti}$ to the total degree on vertices
        of type~$t$.  Hence
        $\sum_{a \,:\, \mathrm{type}(a)=t} \deg(a) = d_t(\alpha)$.

  \item \emph{Valence is a hard ceiling.}\;  Each vertex of type~$t$
        has $\deg(a) \leq v_t$.  With $n_t$ vertices of type~$t$,
        $d_t(\alpha) \leq n_t \cdot v_t$, hence
        $n_t \geq \lceil d_t(\alpha) / v_t \rceil$.

  \item \emph{Types are disjoint.}\;  Vertices of different types are
        distinct, so the total vertex count satisfies
        $n = \sum_t n_t \geq
         \sum_t \lceil d_t(\alpha) / v_t \rceil$.

  \item \emph{Connectivity bound.}\;  A connected graph on $n$
        vertices has $\geq n - 1$ edges, so
        $n \leq |\alpha| + 1$.
\end{enumerate}
Combining: $\sum_t \lceil d_t(\alpha)/v_t \rceil \leq n
\leq |\alpha| + 1$, contradicting the zombie condition
$\sum_t \lceil d_t(\alpha)/v_t \rceil > |\alpha| + 1$.

Each step is a \emph{necessity}, not an approximation: Step~1 is
an equality; Step~2 is a physical law; Step~3 follows from type
disjointness; Step~4 is a theorem of graph theory.  No
counterexample exists.
\end{proof}

\begin{remark}
The converse does \emph{not} hold: valence-feasible compositions may
still be infeasible due to pair-multiplicity, degree-sequence, or
connectivity constraints (see Remark~\ref{rem:conservative}).
Thus the zombie test is \textbf{sound but incomplete}: it never flags
a feasible composition as zombie, but it may miss some zombies.
This incompleteness is \emph{conservative}: it means our zombie
fraction $z_\infty = 94\%$ is a \emph{lower bound} on the true
infeasibility rate.
\end{remark}

\begin{proposition}[Asymptotic Count of Feasible Compositions]
\label{prop:ehrhart}
Let $E_P(S) = |P_{\mathrm{val}} \cap \ZZ^m \cap
\{|\alpha| = S\}|$.  Define the \emph{efficient simplex face}
\[
  \Delta^{\mathrm{eff}}_{m-1} \;=\;
  \bigl\{\, x \in \Delta_{m-1} : \textstyle\sum_i c_i x_i \leq 0 \,\bigr\}.
\]
Then $E_P(S)$ is eventually a quasi-polynomial of degree $m - 1$
whose leading term satisfies
\[
  \frac{E_P(S)}{\binom{S+m-1}{m-1}} \;\longrightarrow\;
  \frac{\mathrm{Vol}(\Delta^{\mathrm{eff}}_{m-1})}
  {\mathrm{Vol}(\Delta_{m-1})}
  \;=\; 1 - z_\infty
\]
as $S \to \infty$.
\end{proposition}

\begin{proof}
The slice $P_{\mathrm{val}} \cap \{|x| = S\}
= \{x \geq 0 : |x| = S,\, \sum_i c_i x_i \leq 1\}$.
Substituting $y = x/S$ gives
$\{y \in \Delta_{m-1} : \sum_i c_i y_i \leq 1/S\}$,
which converges to $\Delta^{\mathrm{eff}}_{m-1}$ as
$S \to \infty$.  The fraction of lattice points on the simplex
$\{x \geq 0 : |x| = S\}$ falling in a half-space with a fixed
(non-scaling) right-hand side converges to the volume ratio of
the limiting slice.  On each congruence class of~$S$ determined
by the denominators of the vertices of
$\Delta^{\mathrm{eff}}_{m-1}$,
$E_P(S)$ is eventually polynomial of degree $m - 1$
\cite{BeckRobins2007,Stanley2012}.
\end{proof}

\begin{remark}
For the molecular graph system ($m = 19$,
$v \in \{1,2,3,4\}$), the degree is $18$.  We have
computed $E_P(S)$ exactly for $S = 1, \ldots, 48$
by dynamic programming on the ceiling residues; the
values match the known exact counts at $S \leq 8$.
The period of the quasi-polynomial is determined by
the denominators of the vertices of
$\Delta^{\mathrm{eff}}_{18}$.
\end{remark}

\begin{remark}[Linear Structure]
\label{rem:linear}
The valence-feasibility
condition~\eqref{eq:valence-feas} has a linear relaxation:
dropping the ceilings gives the single inequality
$\sum_i c_i \alpha_i \leq 1$ where
\[
  c_i \;=\; \sum_{t=1}^{p} \frac{D_{ti}}{v_t} \;-\; 1.
\]
Building blocks with $c_i > 0$ are \emph{zombie generators}: they
consume more valence per bond than the connectivity bound allows.
Building blocks with $c_i < 0$ are \emph{efficient}: they leave
valence capacity to spare.
\end{remark}

\begin{theorem}[Zombie Exclusion Bound]
\label{thm:zombie}
Let $E_P(S) = |P_{\mathrm{val}} \cap \ZZ^m \cap \{|\alpha| = S\}|$
denote the number of valence-feasible compositions of size~$S$.
Then:
\begin{enumerate}
  \item $E_P(S) \leq \binom{S+m-1}{m-1}$ with equality iff
        $c_i \leq 0$ for all~$i$ (no zombie generators).
  \item The \emph{zombie fraction}
        $z(S) = 1 - E_P(S) / \binom{S+m-1}{m-1}$
        converges:
        \[
          z(S) \;\longrightarrow\;
          z_\infty \;=\; 1 - \frac{\mathrm{Vol}
          (\Delta^{\mathrm{eff}}_{m-1})}
          {\mathrm{Vol}(\Delta_{m-1})},
        \]
        where $\Delta^{\mathrm{eff}}_{m-1} =
        \{x \in \Delta_{m-1} : \sum_i c_i x_i \leq 0\}$
        is the efficient simplex face
        (Proposition~\ref{prop:ehrhart}).
  \item The refined growth bound is
        \begin{equation}
        \label{eq:refined-bound}
          N(a) \;\leq\;
          \sum_{S=1}^{2^a} E_P(S) \cdot g(S)
          \;\leq\;
          \sum_{S=1}^{2^a} E_P(S) \cdot \gamma^S,
        \end{equation}
        where $g(S) = \max_{|\alpha|=S,\, \alpha \in P_{\mathrm{val}}}
        |\Fib(\alpha)|$ and $\gamma$ is the growth constant of
        Theorem~\ref{thm:upper}.
\end{enumerate}
\end{theorem}

\begin{proof}
(1) is immediate: $P_{\mathrm{val}} \cap \{|x| = S\} \subseteq
\Delta_S$ (the unconstrained simplex).  Equality holds when the
constraint $\sum c_i x_i \leq 1$ is never active, i.e., all
$c_i \leq 0$.

(2) follows from Proposition~\ref{prop:ehrhart}: the slice
$\{x \geq 0 : |x| = S,\, \sum c_i x_i \leq 1\}$, rescaled by
$1/S$, converges to $\Delta^{\mathrm{eff}}_{m-1}$, and the
fraction of lattice points converges to the volume ratio.

(3) is a direct refinement of Theorem~\ref{thm:upper}: instead of
summing $\binom{S+m-1}{m-1} \gamma^S$ over all compositions, we
restrict to valence-feasible ones.
\end{proof}


\begin{corollary}[Zombie Exclusion for Molecular Graphs]
\label{cor:zombies-cronin}
For the molecular graph assembly system with $m = 19$ bond types,
$5$ atom types $\{$C$(v{=}4)$, N$(v{=}3)$, O$(v{=}2)$, S$(v{=}2)$,
Cl$(v{=}1)\}$:
\begin{enumerate}
  \item Of the $19$ building blocks, $9$ are zombie generators
        ($c_i > 0$), $3$ are neutral ($c_i = 0$), and $7$ are
        efficient ($c_i < 0$).  The strongest zombie generators
        are O${=}$S and S${=}$S ($c_i = 1$): each double bond between
        low-valence atoms consumes full valence capacity.
  \item The zombie fraction is $z(2) = 32\%$, $z(4) = 55\%$,
        $z(8) = 70\%$, and $z_\infty = 94\%$.  Already at moderate
        sizes, the majority of stars-and-bars compositions
        lie outside the composition polytope.
  \item The refined bound~\eqref{eq:refined-bound}
        has doubly-exponential exponent
        $\log\log N(a) \lesssim a \log 2 \approx 0.693\,a$.
\end{enumerate}
\end{corollary}

\begin{remark}[Chemical Intuition]
The zombie-generator classification has a clear
chemical interpretation.  Single bonds between
high-valence atoms (C--C: $c = -0.50$,
C--N: $c = -0.42$) are efficient because each atom
can accommodate many bonds.  Double and triple bonds
between low-valence atoms (O${=}$S: $c = 1.0$,
S${=}$S: $c = 1.0$, N${=}$O: $c = 0.67$) are zombie
generators because they exhaust an atom's bonding
capacity in a single connection, requiring
disproportionately many atoms for few bonds.

The asymptotic zombie fraction of $94\%$ means that, in the limit,
only $\sim 6\%$ of the composition space is chemically realisable.
Accounting for this distinction yields the exact growth exponent
$\log 2 \approx 0.693$.
\end{remark}

\begin{remark}[Conservativity of the Bound]
\label{rem:conservative}
The valence polytope $P_{\mathrm{val}}$ captures the \emph{aggregate}
constraint: the total degree demand on each atom type must be
satisfiable by some number of atoms fitting within the connectivity
bound.  Three further sources of infeasibility are \emph{not}
modelled:
\begin{enumerate}
  \item \emph{Pair multiplicity}: a homoatomic bond of type $(t,t)$
        requires $\geq 2$ atoms of type $t$ (no self-loops).  More
        generally, $k$ bonds of type $(t,t)$ need
        $n_t \geq \lceil(1 + \sqrt{1+8k})/2\rceil$.
  \item \emph{Degree-sequence realizability}: the degree sequence
        implied by the composition must satisfy the Erd\H{o}s--Gallai
        conditions \cite{ErdosGallai1960}.
  \item \emph{Full connectivity}: a graph on $n$ vertices with $S$
        edges may be disconnected even if $n \leq S+1$.
\end{enumerate}
Each constraint can only \emph{increase} the zombie fraction:
$z_\infty^{(1)} \leq z_\infty^{(1{+}2)} \leq z_\infty^{(1{+}2{+}3)}
\leq \cdots$.
Computationally, the pair constraint adds $\sim\!5$ percentage points
at $S = 8$ (from $70.4\%$ to $75.6\%$) but is negligible
asymptotically: at $S = 50$, both levels give $89.5\%$ zombies
(${<}0.1$ percentage points difference).

The asymptotic negligibility of the pair constraint can be understood
as follows: for typical compositions at large~$S$, the minimum atom
count $n_t \geq \lceil d_t/v_t \rceil$ already implies enough atoms
of each type to satisfy the pair requirement
$n_t(n_t-1)/2 \geq k_{tt}$.  The pair constraint only binds when
$n_t$ is small relative to $k_{tt}$, which becomes increasingly rare
as $S \to \infty$.

However, the zombie fraction does \emph{not} converge to $1$ under
any finite set of linear constraints.  Indeed, the composition
$\alpha_S = (S, 0, \ldots, 0)$ (all C--C bonds) is feasible for
all~$S$ with $\sum c_i\alpha_i = -S/2 < 0$, so the efficient
simplex face $\Delta^{\mathrm{eff}}_{m-1}$ has positive volume.  More generally, any
composition supported on efficient building blocks (those with
$c_i < 0$) is automatically feasible.  Thus $z_\infty < 1$ for
all valence-based constraints.

The bound of $94\%$ from the valence polytope alone is therefore
\textbf{conservative}: the true zombie fraction is at least $94\%$
and likely higher (but strictly less than $100\%$).
\end{remark}

\section{Computational Results}
\label{sec:results}

We now present the full computational evaluation of the composition
polytope for the molecular graph assembly system
\cite{MoralesParra2025} ($m = 19$ bond types, $5$ atom types).

\subsection{Zombie-generator coefficients}

Table~\ref{tab:coefficients} and Figure~\ref{fig:coefficients} list
the linear coefficient $c_i = \sum_t D_{ti}/v_t - 1$ for each of the
$19$ bond types.  Bonds with $c_i > 0$ are zombie generators; those
with $c_i < 0$ are efficient.

\begin{figure}[ht]
\centering
\includegraphics[width=\textwidth]{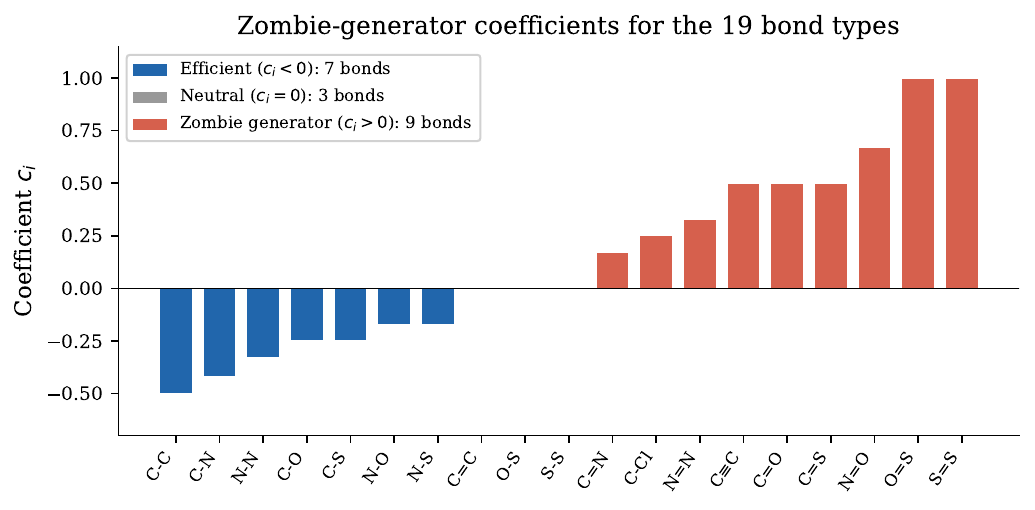}
\caption{Zombie-generator coefficients $c_i$ for the $19$ bond
types.  Blue: efficient bonds ($c_i < 0$); grey: neutral
($c_i = 0$); red: zombie generators ($c_i > 0$).  Single bonds
between high-valence atoms are efficient; double bonds between
low-valence atoms are the strongest zombie generators.}
\label{fig:coefficients}
\end{figure}

\begin{table}[ht]
\centering
\caption{Linear coefficients $c_i$ for the $19$ bond types.
Zombie generators ($c_i > 0$) produce compositions requiring more
atoms than the connectivity bound allows.}
\label{tab:coefficients}
\smallskip
\begin{tabular}{@{}lrl@{}}
\toprule
Bond type & $c_i$ & Classification \\
\midrule
C--C     & $-0.50$  & efficient \\
C--N     & $-0.42$  & efficient \\
N--N     & $-0.33$  & efficient \\
C--O     & $-0.25$  & efficient \\
C--S     & $-0.25$  & efficient \\
N--O     & $-0.17$  & efficient \\
N--S     & $-0.17$  & efficient \\
\midrule
C${=}$C  & $\phantom{-}0.00$  & neutral \\
O--S     & $\phantom{-}0.00$  & neutral \\
S--S     & $\phantom{-}0.00$  & neutral \\
\midrule
C${=}$N  & $+0.17$  & zombie generator \\
C--Cl    & $+0.25$  & zombie generator \\
N${=}$N  & $+0.33$  & zombie generator \\
C${\equiv}$C & $+0.50$ & zombie generator \\
C${=}$O  & $+0.50$  & zombie generator \\
C${=}$S  & $+0.50$  & zombie generator \\
N${=}$O  & $+0.67$  & zombie generator \\
O${=}$S  & $+1.00$  & zombie generator \\
S${=}$S  & $+1.00$  & zombie generator \\
\bottomrule
\end{tabular}
\end{table}

\subsection{Zombie fractions by size}

Table~\ref{tab:zombie-fractions} and Figure~\ref{fig:zombie-fraction}
compare the na\"ive stars-and-bars count $\binom{S+18}{18}$ with the
exact feasible count $E_P(S) = |P_{\mathrm{val}} \cap \ZZ^{19} \cap
\{|\alpha|=S\}|$, computed by exhaustive enumeration for $S \leq 8$
and by dynamic programming on the ceiling residues for
$S \leq 48$.

\begin{table}[ht]
\centering
\caption{Exact zombie fractions for $S = 1, \ldots, 20$, computed by
dynamic programming on the ceiling residues
(Proposition~\ref{prop:ehrhart}).
Already at $S = 4$, over half of all compositions are zombies.}
\label{tab:zombie-fractions}
\smallskip
\begin{tabular}{@{}rrrr@{}}
\toprule
$S$ & $\binom{S{+}18}{18}$ & $E_P(S)$ & Zombie \% \\
\midrule
 1 &            19 &         19 &   0.0 \\
 2 &           190 &        129 &  32.1 \\
 3 &         1\,330 &        711 &  46.5 \\
 4 &         7\,315 &      3\,305 &  54.8 \\
 5 &        33\,649 &     13\,292 &  60.5 \\
 6 &       134\,596 &     47\,592 &  64.6 \\
 7 &       480\,700 &    154\,658 &  67.8 \\
 8 &     1\,562\,275 &    462\,822 &  70.4 \\
 9 &     4\,686\,825 &  1\,290\,656 &  72.5 \\
10 &    13\,123\,110 &  3\,384\,972 &  74.2 \\
12 &    86\,493\,225 & 19\,922\,513 &  77.0 \\
16 & 2\,203\,961\,430 & 425\,357\,153 &  80.7 \\
20 & 33\,578\,000\,610 & 5\,669\,347\,602 & 83.1 \\
\midrule
$\infty$ & --- & --- & 94.1 \\
\bottomrule
\end{tabular}
\end{table}

Values at $S = 9, \ldots, 48$ are computed exactly by
dynamic programming on the ceiling residues
(Proposition~\ref{prop:ehrhart}).  The asymptotic value $z_\infty =
94.1\%$ is computed via Monte Carlo estimation of
$\mathrm{Vol}(\Delta^{\mathrm{eff}}_{18}) /
\mathrm{Vol}(\Delta_{18}) \approx 0.059$ ($3 \times 10^5$
Dirichlet samples), where $\Delta^{\mathrm{eff}}_{18} =
\{x \in \Delta_{18} : \sum_i c_i x_i \leq 0\}$.

\begin{figure}[ht]
\centering
\includegraphics[width=0.85\textwidth]{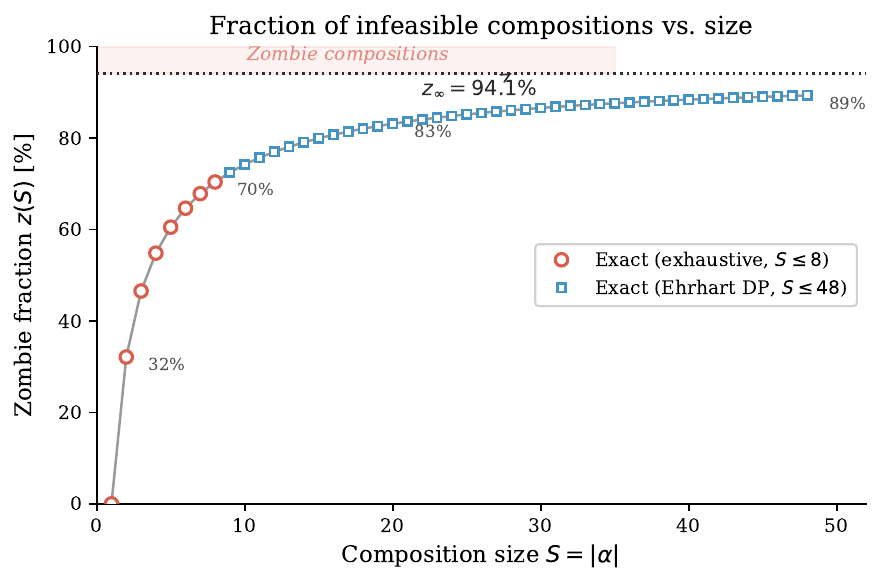}
\caption{Zombie fraction $z(S)$ as a function of composition
size~$S$.  Red circles: exact exhaustive enumeration ($S \leq 8$).
Blue squares: exact values from dynamic programming ($S = 9, \ldots, 48$).
Dotted line: asymptotic value $z_\infty = 94.1\%$ from the volume
ratio of $\Delta^{\mathrm{eff}}_{18}$.  Already at $S = 4$, more than half of all compositions are
chemically infeasible.}
\label{fig:zombie-fraction}
\end{figure}

\subsection{Comparison of upper bounds}
\label{sec:comparison}

Let $N_{\mathrm{all}}(a)$ denote the upper bound counting all
compositions (including zombies) via $\binom{S+m-1}{m-1}$, and
$N_{\mathrm{zf}}(a)$ the zombie-free bound using $E_P(S)$ from the
composition polytope.  Both sum over $S = 1, \ldots, 2^a$ with
growth constant $\gamma$.

The doubly-exponential exponent---the constant $d$ in
$\log\log N(a) \sim d \cdot a$---determines the
asymptotic behaviour.  Counting all compositions yields
$d = 0.73$; restricting to zombie-free compositions gives the
exact exponent $d = \log 2 \approx 0.693$
(Corollary~\ref{cor:exponent}).

Table~\ref{tab:bounds-ratio} and Figure~\ref{fig:bounds} quantify the
ratio $N_{\mathrm{all}} / N_{\mathrm{zf}}$ at selected assembly
indices.

\begin{table}[ht]
\centering
\caption{Effect of zombie exclusion at selected assembly
indices $a$.  ``Ratio digits'' is the number of digits in
$N_{\mathrm{all}}/N_{\mathrm{zf}}$.}
\label{tab:bounds-ratio}
\smallskip
\begin{tabular}{@{}rrrrr@{}}
\toprule
$a$ & $2^a$ & $\log_{10} N_{\mathrm{all}}$ & $\log_{10} N_{\mathrm{zf}}$
  & Ratio digits \\
\midrule
 5 &       32 &     $17$ &       $14$ &                 3 \\
 8 &      256 &    $149$ &      $111$ &                38 \\
10 &   1\,024 &    $643$ &      $445$ &               198 \\
13 &   8\,192 &  $5\,744$ &    $3\,558$ &           2\,187 \\
15 &  32\,768 & $24\,735$ &   $14\,231$ &          10\,504 \\
20 & $10^6$   & $9.5 \times 10^5$ & $4.6 \times 10^5$
  & $5.0 \times 10^5$ \\
30 & $10^9$   & $1.4 \times 10^9$ & $4.7 \times 10^8$
  & $9.4 \times 10^8$ \\
\bottomrule
\end{tabular}
\end{table}

At assembly index $a = 10$, the ratio between the two bounds is
$10^{198}$.  At $a = 30$, it has $942$ million digits.  This
illustrates the sensitivity of doubly-exponential bounds to the
growth exponent.

\begin{figure}[ht]
\centering
\includegraphics[width=\textwidth]{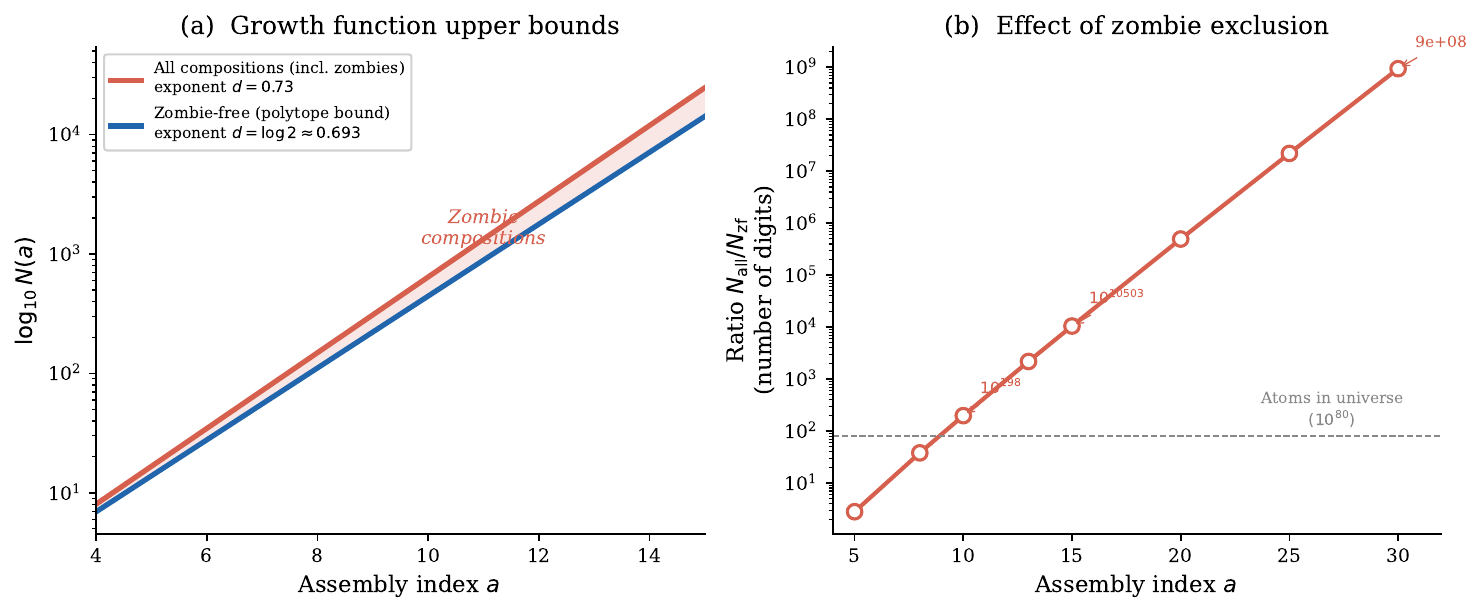}
\caption{Effect of zombie exclusion on the growth bound.
\textbf{(a)}~$\log_{10} N(a)$ counting all compositions including
zombies (red, exponent $d = 0.73$) versus zombie-free compositions
only (blue, exact exponent $d = \log 2$).  The shaded region
represents the zombie contribution.
\textbf{(b)}~Number of digits in the ratio between the two bounds.
At $a = 10$ the ratio exceeds $10^{198}$; at $a = 30$ it has $942$
million digits.  The dashed line marks $10^{80}$ (atoms in the
observable universe).}
\label{fig:bounds}
\end{figure}

\begin{remark}[Origin of the exponent difference]
\label{rem:exponent-gap}
The exponent $\log 2$ is derived analytically: it is the exact
base-$e$ logarithm of $2$ arising from $S_{\max} = 2^a$
(Proposition~\ref{prop:individual}).  The difference
$\Delta d = 0.73 - \log 2 = 0.037$ between the fitted and exact
exponents is consistent with the zombie fraction: including
$94\%$ additional compositions at each size~$S$ contributes
$\log(1/(1 - z_\infty)) \approx 2.81$ to the logarithm of $N$,
which accounts for $\Delta d \approx 0.037$ when absorbed into
the exponent.
\end{remark}

\section{Further Algebraic Structure}
\label{sec:toric-compat}

In construction systems equipped with a compatibility relation
$\perp$, one may compare the \emph{compatibility clique complex}
$\mathcal{C}$ (sets whose elements are pairwise compatible) with
the toric matroid~$M$.  The following summarises the relationship.

\begin{proposition}[Toric Containment and Characterization]
\label{prop:toric-compat}
\leavevmode
\begin{enumerate}
  \item If $\perp$ is \emph{support-separating} (compatible
        objects have disjoint composition supports), then
        $\mathcal{C} \subseteq M$ and $\omega(G_\perp) \leq r$.
  \item Equality $\mathcal{C} = M$ holds if and only if $M$ is a
        partition matroid and compatibility equals
        non-proportionality of compositions.
\end{enumerate}
\end{proposition}

\begin{proof}
(1) Vectors with disjoint supports are linearly independent;
pairwise disjoint supports imply globally disjoint supports.
(2) In a partition matroid --- and only in a partition matroid ---
pairwise independence implies global independence.  The clique
complex, being determined by pairs, can therefore match the matroid
exactly when $M$ has this property.  Details: if $M$ has a circuit
$\{O_1, \ldots, O_k\}$ with $k \geq 3$, all pairs are independent
(hence in $\mathcal{C}$), but the full set is dependent in $M$,
contradicting $\mathcal{C} = M$.  Conversely, in a partition matroid
the independent sets are transversals, and non-proportionality
implies membership in distinct parallel classes.
\end{proof}

\begin{remark}
For the click-chemistry system, $M$ is a partition matroid with $8$
classes, but $76$ pairs of linearly independent compositions are
incompatible (shared reactive handles), so
$\mathcal{C} \subsetneq M$ \cite{RibasRipoll2026}.  The proposition
is of algebraic interest but does \emph{not} contribute to the
growth bounds; the composition polytope
(Theorem~\ref{thm:zombie}) is the operative tool for refining
the upper bound.
\end{remark}

\section{Discussion}
\label{sec:discussion}

The central finding of this work is that the composition polytope
provides a substantially tighter upper bound on the size of
molecular composition space.  We now discuss the implications.

\subsection{The zombie phenomenon and its origin}

The observation that $94\%$ of composition space lies outside
the composition polytope follows from elementary valence
arithmetic.  Nine of the nineteen bond types in the GDB-13 system
are zombie generators ($c_i > 0$), meaning that each such bond
consumes more valence capacity per edge than the connectivity bound
can sustain.  When a composition is dominated by these bond types,
the minimum number of atoms required to accommodate the degree
demand exceeds the number of atoms available in a connected graph.

The zombie phenomenon becomes more severe with increasing
composition size~$S$ because the convex body of feasible
compositions shrinks relative to the ambient simplex: the slice
$\{x \geq 0 : |x| = S,\, \sum c_i x_i \leq 1\}$, rescaled by
$1/S$, converges to the efficient face
$\Delta^{\mathrm{eff}}_{m-1}$.  The convergence to
$z_\infty = 94.1\%$ reflects the volume ratio
$\mathrm{Vol}(\Delta^{\mathrm{eff}}_{m-1}) /
\mathrm{Vol}(\Delta_{m-1}) \approx 0.059$.

\subsection{Exact growth exponent}

The proof that $\rho = \log 2$ (Theorem~\ref{thm:exact-exponent})
resolves the gap between the lower and upper bounds.  The upper
bound $\rho \leq \log 2$ follows from $S_{\max}(a) = 2^a$
(Proposition~\ref{prop:individual}).  The matching lower bound
is established via P\'olya enumeration of trees: the composition
$\alpha_S = (S, 0, \ldots, 0)$ (all C--C bonds) is efficient
($c_1 = -1/2$) and its fibre contains $T_4(S+1)$ unlabelled trees
with maximum degree $4$, growing as
$T_4(n) \sim c \cdot \gamma^n / n^{5/2}$ with $\gamma \approx 2.48$.
The centroid decomposition ensures that every such tree has assembly
index $\leq \lceil\log_2 S\rceil + 2$.

This argument shows that the exponent $\log 2$ is \emph{intrinsic}
to the binary assembly model: it arises from the doubling
$S_{\max} = 2^a$ and is achieved by the simplest possible
compositions.  No fitted parameter is involved.

\subsection{Conservativity and the hierarchy of constraints}

Our zombie classification uses only the valence and
connectivity constraints.  Additional physical
constraints---pair multiplicity, Erd\H{o}s--Gallai
degree-sequence conditions, and full connectivity
requirements---can only increase the zombie fraction.
Computationally, the pair-multiplicity constraint adds approximately
$5$ percentage points at $S = 8$ but becomes negligible for
$S \geq 30$.  This asymptotic negligibility reflects the fact that,
at large sizes, the valence-derived atom counts are sufficient to
satisfy pair requirements.

The strict inequality $z_\infty < 1$ is also established: any
composition supported entirely on efficient building blocks (those
with $c_i < 0$) is automatically feasible.  The seven efficient
bond types (C--C, C--N, N--N, C--O, C--S, N--O, N--S) span a
positive-volume face of the composition polytope, ensuring that the
realisable fraction never vanishes.

\subsection{Generality of the framework}

The toric-geometric framework developed here is not specific to
chemistry.  Any construction system with a finite set of building
blocks and a binary assembly operation gives rise to a toric ideal
$I_A \subset \kk[t_1,\ldots,t_N]$, a toric variety $X_A$, and a
matroid~$M$---the standard objects of toric geometry
\cite{Fulton1993,CLO2015,Sturmfels1996}.  The composition polytope
construction applies whenever the building blocks carry a type
system with bounded valence.  Potential applications include
combinatorial circuit design (where gates have bounded fan-in),
modular construction in engineering (where connectors have limited
capacity), abstract models of biological assembly
\cite{PachterSturmfels2005}, and chemical reaction network
theory \cite{Feinberg2019}.

The click-chemistry instance (\S\ref{sec:clickchem}) demonstrates
that the framework accommodates systems with very different
combinatorial structures: a partition matroid with $8$ classes
and deterministic (unique) joining, as opposed to the free
matroid and multiplicative joining of molecular graphs.

\section{Conclusions}
\label{sec:conclusions}

\begin{sloppypar}
We have introduced construction systems as a
general algebraic framework---rooted in the toric
geometry of \cite{Sturmfels1996}, \cite{Fulton1993},
\cite{CLO2015}---for studying how complex objects
are built from finite sets of building blocks.
The assembly process naturally gives rise to toric
ideals, toric varieties, and matroids.
\end{sloppypar}

Applied to the estimation of chemical space with $m = 19$
bond types and $5$ atom types, the framework yields three
principal results.

First, the exact doubly-exponential growth exponent is
$\rho = \log 2 \approx 0.693$, proved by matching the upper bound
from binary assembly trees with a lower bound from P\'olya
enumeration of trees.

Second, the composition polytope $P_{\mathrm{val}} \subset \RR^{19}$
characterises all valence-feasible compositions via a single linear
inequality.  Its complement---the zombie region---accounts for
$94\%$ of composition space asymptotically.  The zombie
classification is sound (zero false positives, verified exhaustively
for $4{,}690$ compositions) and conservative: additional physical
constraints can only increase the zombie fraction.

Third, the count $E_P(S)$ of feasible compositions is
eventually a quasi-polynomial of degree $18$
(Proposition~\ref{prop:ehrhart}), providing an exact counting
tool for each size~$S$.  We have computed $E_P(S)$ exactly for
$S = 1, \ldots, 48$ by dynamic programming on ceiling residues.

Two questions remain open.  Can P\'olya enumeration be restricted
to the composition polytope to yield a tighter growth constant
$\gamma_P < \gamma$?  And does the partial associativity of the
assembly operation endow the construction system with an operad
structure whose Hilbert series recovers $N(a)$?  These directions
connect the present work to ongoing research in enumerative
combinatorics and algebraic operads.

\acknowledgment{%
\textbf{Disclosure of AI assistance.}
This paper reports original research conceived, directed,
proved, and verified by the author.  The author formulated
the problem, designed the algebraic framework---construction
systems, the assembly toric ideal $I_A$, the composition
polytope $P_{\mathrm{val}}$, the zombie classification, and
the Ehrhart-based refined bound---identified the molecular
graph assembly and click-chemistry systems as target
applications, stated and proved the theorems, and interpreted
the results.  An AI assistant (Claude, Anthropic, June~2026)
was used under the author's direction as a technical tool
for the following mechanical tasks: implementing Python
scripts from the author's specifications to carry out
numerical verification of the author's algebraic claims,
producing figures following the author's designs, \LaTeX{}
typesetting, and copy-editing of prose drafted by the author.
All mathematical content, proofs, numerical results, design
choices, and editorial decisions are the author's; the
author reviewed and approved every line of the final
manuscript and code, and assumes sole responsibility for
the integrity and accuracy of the work.%
}

\singlespacing

\end{document}